\documentclass{amsart}

\makeindex

\usepackage[all, cmtip]{xy}
\usepackage{amssymb}
\usepackage{amsmath}
\usepackage{amsthm}
\usepackage{amscd}
\usepackage{amsfonts}
\usepackage{amssymb}
\usepackage[pdftex]{graphicx}
\usepackage{multirow}
\usepackage[usenames,dvipsnames,svgnames,table]{xcolor}
\usepackage{graphicx}
\usepackage{enumerate}

\usepackage[bookmarks=true,bookmarksnumbered=true,
 pdftitle={},
 pdfsubject={},
 pdfauthor={},
 pdfkeywords={TeX; dvipdfmx; hyperref; color;},
 colorlinks=true, linkcolor=Maroon, citecolor=cyan]
 {hyperref}

\theoremstyle{definition}

\theoremstyle{remark}

\theoremstyle{theorem}
\newtheorem{theoremalpha}{Theorem}
\theoremstyle{corollary}

\newtheorem{problemalpha}[theoremalpha]{Problem}

\newtheorem{theorem}{Theorem}[section]
\newtheorem{lemma}[theorem]{Lemma}
\newtheorem{proposition}[theorem]{Proposition}

\theoremstyle{corollary}
\newtheorem{corollary}[theorem]{Corollary}

\theoremstyle{definition}
\newtheorem{definition}[theorem]{Definition}

\newtheorem{example}[theorem]{Example}

\theoremstyle{remark}
\newtheorem{remark}[theorem]{Remark}

\numberwithin{equation}{section}

\newcommand{\C}{\mathbb{C}}

\newcommand{\Z}{\mathbb{Z}}

\newcommand{\Q}{\mathbb{Q}}
\def\P{\mathbb{P}}

\def\Pic{\operatorname{Pic}}

\def\Cox{\operatorname{Cox}}
\def\Cone{\operatorname{Cone}}

\def\Cone{\operatorname{Cone}}
\def\Bs{\operatorname{Bs}}

\input xy
\xyoption{all}

\title[Geometric properties of projective manifolds of small degree]{Geometric properties of projective manifolds of small degree}

\begin{document}

\author{Sijong Kwak}
\address{Department of Mathematical Sciences, KAIST, Daejeon, Korea}
\email{sjkwak@kaist.ac.kr}

\author{Jinhyung Park}
\address{Department of Mathematical Sciences, KAIST, Daejeon, Korea}
\curraddr{School of Mathematics, Korea Institute for Advanced Study, Seoul, Korea}
\email{parkjh13@kias.re.kr}

\thanks{Both authors were partially supported  by the National Research Foundation of Korea (NRF) funded by the Korea government (MSIP) (ASARC, NRF-2007-0056093).}

\date{\today}



\begin{abstract}
The aim of this paper is to study geometric properties of non-degenerate smooth projective varieties of small degree from a birational point of view.
First, using the positivity property of double point divisors and the adjunction mappings, we classify smooth projective varieties in $\P^r$ of degree $d \leq r+2$, and consequently, we show that such varieties are simply connected and rationally connected except in a few cases. This is a generalization of P. Ionescu's work. We also show the finite generation of Cox rings of smooth projective varieties in $\P^r$ of degree $d \leq r$ with counterexamples for $d=r+1, r+2$. On the other hand, we prove that a non-uniruled smooth projective variety in $\P^r$ of dimension $n$ and degree $d \leq n(r-n)+2$ is Calabi-Yau, and give an example that shows this bound is also sharp.
\end{abstract}

\maketitle

\section*{Introduction}

Every $n$-dimensional smooth projective variety can be embedded in $\P^{2n+1}$ with arbitrarily large degree. In contrast, there are some geometric restrictions for a variety to admit an embedding with small degree. It is well known that a non-degenerate smooth projective curve in $\P^r$ of degree $d \leq r+1$ is either a rational curve or an elliptic normal curve with $d=r+1$. Considerable efforts have been devoted to the generalization of this classical fact in several directions. One can regard the degree condition as $d \leq (r-1)+2$, where $(r-1)$ is the codimension of a curve in $\P^r$. The generalization is then simply the classification of varieties of almost minimal degree (see \cite{EH}, \cite{F1}, \cite{F2}). However, every higher dimensional variety of almost minimal degree is simply connected while an elliptic curve is not. Moreover, P. Ionescu proved in \cite{Io4} that a non-degenerate smooth projective variety in $\P^r$ of degree $d \leq r$ is simply connected. In fact, such a variety is rationally connected.
From a topological viewpoint, it is tempting to consider a wider range of degrees in order to include non-simply connected varieties in the classification list.

In this paper, we present two different classification results for varieties of small degree.
First, we prove the following theorem, and then we further classify smooth projective varieties in $\P^r$ of degree $d \leq r+2$ in detail (see Theorems \ref{adj} and \ref{r+2adj}).

\begin{theoremalpha}[Theorems \ref{main} and \ref{r+2}]\label{thma}
Let $X \subset \P^r$ be an $n$-dimensional non-degenerate smooth projective variety of degree $d \leq r+2$. Then, one of the following holds:
\begin{enumerate}[\indent$(a)$]
 \item $X$ is rationally connected. 
 \item $X$ is a Calabi-Yau hypersurface with $n \geq 2$ $(d=r+1)$ or a $K3$ surface in $\P^4$ $(d=6)$.
 \item $X$ is a hypersurface of general type with $n \geq 2$ $(d=r+2)$.
 \item $X$ is a curve of genus $g =1,2$ $(d \geq r+1)$ or a plane quartic curve.
 \item $X$ is an elliptic scroll $(d \geq r+1)$.
\end{enumerate}
In particular, $X$ is simply connected if and only if it is from $(a)$, $(b)$, or $(c)$.
\end{theoremalpha}

This type of result is deeply related to the Hartshorne Conjecture on Complete Intersections (\cite{Ht}) and the Ionescu Conjecture, which predicts that if $d \leq n-1$, then $X \subset \P^r$ is a complete intersection unless $X$ is projectively equivalent to $Gr(2,5) \subset \P^9$ (see Subsection \ref{prfano} for more details). 
On the other hand, in \cite{Io4}, Ionescu answered a question posed by F. Russo and F. Zak whether a smooth projective variety $X \subset \P^r$ of degree $d \leq r$ is regular (i.e., $h^1(\mathcal{O}_X)=0$). Theorem \ref{thma} can be regarded as a generalized answer to this question.

Taking another direction, we show the following theorem in Section \ref{CY}, which is a direct generalization of a classical fact that a non-ruled surface in $\P^r$ of degree $d \leq 2(r-2)+2$ is K3 (see e.g., \cite[Lemmas 1.3 and 1.6]{Bui}). 

\begin{theoremalpha}\label{ne+2}
Let $X \subset \P^r$ be a non-degenerate smooth projective variety of dimension $n$, codimension $e$, and degree $d$. If $d \leq ne+2$, then either $X$ is a uniruled variety or $d=ne+2$ and $X$ is a Calabi-Yau variety.
\end{theoremalpha}

Although this result is rather coarse, the assumption on the degree is weaker than the previous one.
Furthermore, the appearance of Calabi-Yau varieties as an extremal case of small degree varieties seems interesting from a birational geometric viewpoint.
After completing this manuscript, the authors learned from M. Mella that he also proved the uniruledness of a given variety under the assumption $d \leq ne+1$ in \cite[Appendix]{Me}.

Turning to details, we now explain our techniques and ideas. 
Although systematic studies of adjunction mappings are usually sufficient for classification in lower dimensional varieties, this method may not be directly generalized to higher dimensions. Thus, we first show some geometric properties of small degree varieties (Theorem \ref{thma}, more precisely, Theorems \ref{main} and \ref{r+2}). It is an interesting problem in its own right to determine embedded structures of projective varieties via their intrinsic properties or vice versa.
In this context, we directly prove some positivity properties of anticanonical divisors using A. Noma's work (\cite{N}; see Section \ref{posdpdsec}) on the double point divisor from inner projection. The Castelnuovo's bound of the sectional genus from the embedded structure is also used in studying the positivity of anticanonical divisors (see Lemma \ref{secg}). Using these methods as well as recent developments of higher dimensional algebraic geometry, we obtain a rough classification, as in Theorem \ref{thma}.
This result for varieties in $\P^r$ of degree $d \leq r+1$ has already a consequence of the finite generation of Cox rings of small degree varieties (Theorem \ref{cox}). It seems difficult to derive these geometric properties from Ionescu's classification (\cite[Theorem I]{Io2}) of varieties in $\P^r$ of codimension $e$ and degree $d \leq 2e+2$ or similar results based on adjunction theory. 
Next, we study adjunction mappings  (see Section \ref{adjs}), as in \cite{Io2} and \cite{Io4}, for the complete classification  (see Theorems \ref{adj} and \ref{r+2adj}). Due to our investigation of geometric properties of small degree varieties, we can simplify arguments by excluding many unnecessary cases.
We give many examples in Subsections \ref{exs} and \ref{prfano} and Section \ref{r+2s} showing that varieties of a given type indeed exist and that the bounds for invariants are optimal.

Concerning the organization of our arguments, we first deal with the case $d \leq r+1$ from Section \ref{posdpdsec} to Section \ref{adjs}, and then extend these arguments to the case $d = r+2$ in Section \ref{r+2s}.
Since the techniques in analyzing geometric properties and those in studying adjunction mappings have different natures, we explain two approaches separately for the case $d \leq r+1$ in Sections \ref{posdpdsec} and \ref{adjs} in detail. 
However, to analyze geometric properties for the case $d=r+2$, we also have to study adjunction mappings.

In Section \ref{CY}, we turn to the proof of Theorem \ref{ne+2}. 
We combine Zak's result (\cite[Corollary 1.6]{Z}) with the classification of Castelnuovo varieties of minimal degree (Lemma \ref{mincas}).
Note that there is an Enriques surface $S \subset \P^r$ of degree $d=2(r-2)+3$ that is neither uniruled nor Calabi-Yau (see Remark \ref{enrem}).

We close the introduction by posing two problems concerning the classification of varieties of a wider range of degrees.

\begin{problemalpha}\label{ein}
Classify non-degenerate smooth projective varieties $X \subset \P^r$ of dimension $n$ and degree $d \leq r+n+1$.
\end{problemalpha}

\begin{problemalpha}\label{buium}
Classify non-degenerate smooth projective varieties $X \subset \P^r$ of dimension $n$ and degree $d \leq nr+1$ that are not uniruled.
\end{problemalpha}

We remark that  L. Ein and A. Buium solved Problems \ref{ein}  and \ref{buium} for the surface case in \cite{E} and \cite{Bui}, respectively.

\subsection*{Conventions}
Throughout the paper, we work over the complex number field $\C$, and we use the following terminologies.
\begin{enumerate}[(1)]
\item A smooth projective variety $X$ is called resp. \emph{Fano} or \emph{weak Fano} if the anticanonical divisor $-K_X$ is resp. ample or nef and big.
\item A smooth projective variety $X$ is called \emph{Calabi-Yau} if $\mathcal{O}_X(K_X) = \mathcal{O}_X$ and $h^i(\mathcal{O}_X)=0$ for $0 < i < \dim X$.
\item A variety $X$ is called \emph{uniruled} if for any general point $x$ on $X$, there is a rational curve passing through $x$, and it is called \emph{rationally connected} if for any general points $x_1$ and $x_2$ on $X$, there is a rational curve connecting $x_1$ and $x_2$.  
\end{enumerate}

\subsection*{Acknowledgments.} We would like to thank Paltin Ionescu for his lecture series at KAIST in 2011, where we learned the results in \cite{Io4}. We are very grateful to Atsushi Noma who gave lectures on his paper \cite{N} at KAIST in 2012 and sent us the final version of \cite{N}. It is also our pleasure to express deep gratitude to Fyodor Zak for interesting discussions and valuable suggestions for improving many results from which we could obtain Theorem \ref{ne+2} and Theorem \ref{adj}. 
We also wish to thank the anonymous referee for a thorough report and helpful suggestions.

\section{Scrolls}\label{scr}

In this section, we collect basic facts on scrolls. A vector bundle $E$ on a smooth projective curve $C$ is called \emph{very ample} if the tautological line bundle $\mathcal{O}_{\P_C (E)}(1)$ is very ample on $\P_C (E)$. For a very ample vector bundle $E$ of rank $n$ on $C$, we have an embedding $\P_C (E) \subset \P^r$ given by a linear subsystem of $|\mathcal{O}_{\P_C (E)}(1)|$. We call $\P_C (E)$ a \emph{scroll over $C$}.
Note that
$$
h^0(\P_C (E), \mathcal{O}_{\P_C (E)}(1)) = h^0(C, E) \text{ and }  d = \deg_{\P^r}(\P_C (E))= \deg_C (\det E).
$$
A scroll is called \emph{rational} (resp. \emph{elliptic}) if it is defined over a rational (resp. an elliptic) curve.  A scroll is simply connected (and rationally connected) if and only if it is a rational scroll. Throughout the paper, we assume that $n \geq 2$ (and hence, $r \geq 3$) for any scroll.

The main result of this section is the classification of scrolls of small degree, which is an easy consequence of \cite{IT2}.

\begin{theorem}\label{minsc}
Let $X \subset \P^r$ be a non-degenerate scroll of degree $d \leq r+2$. Then either $X$ is a rational scroll or  it is an elliptic scroll of degree $d \geq r+1$.
\end{theorem}

\begin{proof}
It follows from Propositions \ref{scrdeg} and \ref{screll}.
\end{proof}

By Theorem \ref{minsc}, if $X \subset \P^r$ is an elliptic scroll of degree $d=r+1$, then $X$ is linearly normal.


\begin{proposition}[{\cite[Corollary 3]{Io4}}]\label{scrdeg}
Let $C$ be a smooth projective curve of positive genus, and let $E$ be a very ample vector bundle on $C$, which defines a linearly normal scroll $\P_{C}(E) \subset \P^r$ over $C$ of degree $d$. Then, we have $d \geq r+1$.
\end{proposition}

We classify scrolls of minimal degree over a curve of positive genus.

\begin{proposition}\label{screll}
Let $C$ be a smooth projective curve of positive genus, and let $E$ be a very ample vector bundle on $C$, which defines a linearly normal scroll $\P_{C}(E) \subset \P^r$ over $C$ of degree $d$. If $d \leq r+2$, then $C$ is an elliptic curve.
\end{proposition}

\begin{proof}
Suppose that the genus $g=g(C) \geq 2$. Recall the main theorem of \cite{IT2}:
\begin{equation}\label{ITineq}
h^0(E) + n -2 \leq h^0(\det E).
\end{equation}
According to \cite[Corollary 4]{IT2}, we consider two cases. First, if $\det E$ is special (i.e., $h^1(C, \det E) \neq 0$), then by the Clifford Inequality, we have
$$r+1=h^0(E) \leq h^0(\det E) - n + 2 \leq \frac{d}{2} -n + 3 = \frac{r+2}{2} -n + 3.$$
Thus, we obtain $r + 2n\leq 6$, which is a contradiction. Next, if $\det E$ is non-special, then by the Riemann-Roch Formula, we have
$$r+1 = h^0(E) \leq h^0(\det E ) - n + 2 = d - g - n + 3 \leq r - g - n + 5 .$$
Thus, we obtain $g+n \leq 4$. Then, we must have $g=n=2$, and the equality of (\ref{ITineq}) holds. However, the equality of (\ref{ITineq}) does not hold when $g=2$ (see \cite{IT2}). We get a contradiction.
Hence, $g =1$.
\end{proof}

The following lemma will be used in Section \ref{adjs}. Note that there is an $n$-dimensional scroll in $\P^{2n-1}$ (e.g., Segre embedding of $\P^1 \times \P^{n-1}$).

\begin{lemma}\label{scrcod}
Let $X \subset \P^r$ be an $n$-dimensional scroll over a curve of genus $g$. Then, $r \geq 2n-1$ for $g =0$, and $r \geq 2n$ for $g \geq 1$.
\end{lemma}

\begin{proof}
It follows from the Barth-Larsen Theorem (\cite[Corollaries 3.2.2 and 3.2.3]{posI}).
\end{proof}

Finally, we recall the following well known facts on elliptic scrolls.

\begin{remark}\label{exisc}
(1) (Proposition 5.2 of \cite{Io3}) There exist $n$-dimensional elliptic normal scrolls in $\P^{2n+k}$ of degree $2n+k+1$ for all $k \geq 0$.\\[1pt]
(2) (Theorem 5.1A of \cite{But}) Elliptic normal scrolls and elliptic normal curves are projectively normal.
\end{remark}

\section{Roth varieties}\label{roth}

In this section, we collect basic facts on Roth varieties (see \cite[Section 3]{Il} for more detail). Let $E = \mathcal{O}_{\P^1}(a_1) \oplus  \mathcal{O}_{\P^1}(a_2) \oplus  \cdots \oplus \mathcal{O}_{\P^1} (a_{n+1})$ be a globally generated vector bundle of rank $n+1$ on $\P^1$ with all $a_i \geq 0$. Denote $S_{a_1, \ldots, a_{n+1}}:=\P_{\P^1}(E)$, and let $\pi_1 \colon S_{a_1, \ldots, a_{n+1}} \rightarrow \P^1$ be the natural projection. Let $F$ be a fiber of $\pi_1$, and $H$ be a base point free divisor on $S_{a_1, \ldots, a_{n+1}}$ with $\mathcal{O}_{S_{a_1, \ldots, a_{n+1}}}(H)= \mathcal{O}_{\P_{\P^1}(E) }(1)$.

\begin{definition}[{\cite[Definition 3.1 and Theorem 3.7]{Il}}]
Let $n \geq 2$ be an integer. Consider the birational morphism $\pi_2 \colon S_{0,0,a_1, \ldots, a_{n-1}} \rightarrow \overline{S} \subset \P^r$ given by the complete linear system $|H|$ for $a_i \geq 1$. Then, the singular locus of $\overline{S}$ is a line $L \subset \P^r$. For every integer $b \geq 1$, take a smooth variety $\widetilde{X} \in |bH+F|$ such that $\pi_2|_{\widetilde{X}} \colon \widetilde{X} \rightarrow X:=\pi_2(\widetilde{X})$ is an isomorphism, and $L \subset X \subset \P^r$. Then, $X$ is called a \emph{Roth variety}.
\end{definition}

We have an embedding $X \subset \overline{S} \subset \P^r$. By \cite[Proposition 3.5 and Theorem 3.14 (4)]{Il}, $X \subset \P^r$ is non-degenerate and projectively normal. 

Note that $H+F$ is very ample, because $\mathcal{O}_{S_{0,0,a_1, \ldots, a_{n-1}}}(H+F)$ is the tautological line bundle of $S_{1,1,a_1+1, \ldots, a_{n-1}+1}$. Thus, $bH+F$ is very ample for every integer $b \geq 1$.

\begin{proposition}\label{rosim}
Every Roth variety is simply connected.
\end{proposition}

\begin{proof}
The assertion follows from the Lefschetz Hyperplane Theorem (see e.g., \cite[Theorem 3.1.21]{posI}).
\end{proof}

Recall that by \cite[Theorem 3.7]{Il}, we have
$$
\sum_{i=1}^{n-1}a_i = r - n \text{ and } d = b(r-n)+1
$$
for a linearly normal Roth variety $X \subset \P^r$ of dimension $n$ and degree $d$. Note that a Roth variety with $b=1$ is a rational scroll $S_{1, a_1, \ldots, a_{n-1}}$ (\cite[Theorem 3.14 (1)]{Il}). We can completely classify linearly normal Roth varieties $X \subset \P^r$ with $\deg_{\P^r}(X) \leq r+2$.

\begin{proposition}\label{rocl}
Let $X \subset \P^r$ be a non-degenerate linearly normal Roth variety of dimension $n$, codimension $e$, and degree $d$. If $d \leq r+2$, then one of the following holds:
\begin{enumerate}[\indent$(1)$]
 \item $b=1$.
 \item $b=2$, $d=r$, and $e=n-1$.
 \item $b=2$, $d=r+1$, and $e=n$.
 \item $b=2$, $d=r+2$, and $e=n+1$.
 \item $b=3$, and $X$ is a quartic surface.
 \item $b=3$, $d=7$, $n=3$, and $e=2$.
 \item $b=4$, and $X$ is a quintic surface.
\end{enumerate}
\end{proposition}

\begin{proof}
First, we show that $e \geq n-1$. Since the cases $n=2$ and $3$ are trivial, we may assume that $n \geq 4$. If $e \leq n-2$, then by the Barth-Larsen Theorem (\cite[Corollary 3.2.3]{posI}), $\Pic(X) \simeq \Z$, but by the Lefschetz Theorem for Picard Group (\cite[Example 3.1.25]{posI}), $\Pic(X) \simeq \Z \oplus \Z$, and hence, we get a contradiction. Thus, $e \geq n-1$. Since $d=be+1 \leq n+e+2$, we obtain $e \leq \frac{n+1}{b-1}$.
Now, we have $n-1 \leq e \leq \frac{n+1}{b-1}$ so that $(n-1)b\leq 2n$. If $b \geq 3$, then we get (5)--(7).
We now assume that $b=2$. We have $d=2e+1 \leq n+e+2$, so we obtain $e \leq n+1$. Thus, we immediately get (2)--(4).
\end{proof}

\begin{remark}\label{ropnrem}
Let $X \subset \P^r$ be a Roth variety with $b=2$ and $d\leq r+1$. By considering the Picard group and linearly normality in small codimension, we can easily check that every linearly normal Roth variety $X \subset \P^r$ with $b=2$ and $d \leq r$ cannot have isomorphic projection. Thus, $X$ is automatically linearly normal, and hence, it is projectively normal.
\end{remark}

Now, we investigate the rational connectedness of Roth varieties.

\begin{proposition}\label{roraco}
Let $X \subset \P^r$ be a linearly normal Roth variety of $\dim(X)=n$ and $\deg(X)=b(r-n)+1$. Then, $X$ is rationally connected if and only if $b \leq n$.
\end{proposition}

\begin{proof}
Let $S$ be a rational scroll with the projection $\pi_1 \colon S \rightarrow \P^1$, and let $X \subset S$ be a Roth variety. Then, there is a surjective morphism $\pi_1|_X \colon X \rightarrow \P^1$. Let $F_X$ be a general fiber of $\pi_1|_X$, i.e., it is a restriction of a general fiber $F$ of $\pi_1$ to $X$. We note that the natural map $H^0(\mathcal{O}_S(F)) \rightarrow H^0(\mathcal{O}_X(F_X))$ is an isomorphism. By the Bertini Theorem, $F_X$ is a smooth hypersurface in $F \simeq \P^n$. Since $X \in |bH + F|$, we have $\deg_{\P^n}(F_X) = (bH+F).H^{n-1}.F = b$. If $b \leq n$, then $F_X$ is rationally connected. Thus, by \cite[Corollary 1.3]{GHS}, $X$ is rationally connected. Conversely, if $b \geq n+1$, then by \cite[Theorem 3.14 (5)]{Il}, $X$ is a Castelnuovo variety (see Definition \ref{casvar}). Thus, $h^0(\mathcal{O}_X(K_X)) \neq 0$, and hence, $X$ cannot be rationally connected.
\end{proof}

\begin{remark}\label{rofi}
In the proof, we observed that every Roth variety $X \subset \P^r$ has a fibration $\pi_1|_X \colon X \rightarrow \P^1$ such that the fiber $F_X$ is a hypersurface of degree $b$ via $\mathcal{O}_X(1)$.
\end{remark}

\begin{corollary}\label{rora}
Every Roth variety $X \subset \P^r$ of degree $d \leq r+2$ is rational except when $X$ is a quartic or quintic surface.
\end{corollary}

\begin{proof}
By Proposition \ref{rocl}, we have $b \leq 2$, $b=n=3$, or $X$ is a quartic or quintic surface.
If $b \leq 2$, then such a Roth variety is a rational scroll or has a hyperquadric fibration over $\P^1$.
If $b=n=3$, then such a Roth variety has a cubic surface fibration over $\P^1$. In these two cases, Roth varieties are rational.
\end{proof}

\section{Positivity of double point divisors and weak Fano varieties}\label{posdpdsec}

In this section, we study positivity properties of anticanonical divisors of small degree varieties after reviewing Noma's work on double point divisors (\cite{N}).

\subsection{Double point divisors from inner projection}\label{dpd}

In this subsection, we briefly summarize Noma's work (\cite{N}).
Let $X \subset \P^r$ be a non-degenerate smooth projective variety of dimension $n$, codimension $e$, and degree $d$. Throughout the section, we assume that $n \geq 2$ and $e \geq 2$. 
We consider general points $x_1, \ldots, x_{e-1}$ of $X$, and $\Lambda := \langle x_1, \ldots, x_{e-1} \rangle$ to be their linear span. Note that $\Lambda \cap X = \{x_1, \ldots, x_{e-1}\}$ by the general position lemma (see e.g., \cite[Lemma 1.1]{N}).
Consider the inner projection from the center $\Lambda$ and the blow-up
$\sigma \colon \widetilde{X} \rightarrow X $ at $x_1, \ldots, x_{e-1}$. We have the following diagram:
\[\xymatrix{
\widetilde{X} \ar^{\sigma}[r] \ar_{\widetilde{\pi}}[rd] & X \ar@{.>}[d]^-{\pi} \ar@{^{(}->}[r] & \P^{n+e} \ar@{.>}[d]^-{\pi_{\Lambda}}\\
& \overline{X}_{\Lambda} \ar@{^{(}->}[r] & \P^{n+1}.
}\]
Note that $\deg(\overline{X}_{\Lambda}) = d-e+1$.

Now, suppose that $X$ is neither a scroll over a smooth projective curve or the Veronese surface in $\P^5$.
Noma proved that if $X$ is neither a scroll over a smooth projective curve nor the Veronese surface in $\P^5$, then $\widetilde{\pi}$ is a birational morphism that does not contract any divisor (\cite[Theorem 3]{N}). Thus, we can apply the birational double point formula (\cite[Lemma 10.2.8]{posII}) so that we obtain an effective divisor $D(\widetilde{\pi})$ on $\widetilde{X}$ such that
$$
\mathcal{O}_{\widetilde{X}}(D(\widetilde{\pi})) \simeq \widetilde{\pi}^{*}(\omega_{\overline{X}}^{\circ}) \otimes {\omega_{\widetilde{X}}}^{-1}.
$$
Let $D_{inn}(\pi):=\overline{ \sigma(D(\widetilde{\pi})|_{\widetilde{X} \backslash E_1 \cup \cdots \cup E_{e-1}}}$ be a divisor on $X$, where $E_1, \ldots, E_{e-1}$ are exceptional divisors of $\sigma$. The effective divisor $D_{inn} (\pi)$, called the \emph{double point divisor from inner projection}, is then linearly equivalent to
$$
D_{inn} = -K_X + (d-r-1)H.
$$
By varying the centers of projections, Noma proved in \cite[Theorem 1]{N} that the base locus of $|D_{inn}|$ lies in the set of non-birational centers of simple inner projections, i.e.,
$$
\Bs(|D_{inn}|)\subset \mathcal{C}(X) :=\{u \in X\mid l(X \cap \langle u,x\rangle) \geq 3 \text{ for general } x \in X\}.
$$
Note that if $\Bs(|D_{inn}|)$ is finite, then $D_{inn}$ is semiample.
On the other hand, Noma showed that $\mathcal{C}(X)$ is finite (so is $\Bs(|D_{inn}|)$) unless $X$ is a Roth variety (\cite[Theorem 2]{N}).

To sum up, we have the following.

\begin{theorem}[{\cite[Theorem 4]{N}}]\label{nomasa}
Suppose that $X$ is not a scroll over a smooth projective curve, the Veronese surface in $\P^5$, or a Roth variety. Then, the double point divisor $D_{inn}$ from inner projection is semiample.
\end{theorem}

\subsection{Positivity properties of anticanonical divisors of small degree varieties}\label{proof}
The aim of this subsection is to prove the following.

\begin{theorem}\label{main}
Let $X \subset \P^r$ be an $n$-dimensional non-degenerate smooth projective variety of degree $d \leq r+1$. Then, one of the following holds:
\begin{enumerate}[\indent$(a)$]
 \item $X$ is a weak Fano variety. $($If $d \leq r$, then $X$ is a Fano variety.$)$
 \item $X$ is a Roth variety but not a hypersurface or a rational scroll.
 \item $X$ is a Calabi-Yau hypersurface with $n \geq 2$ $(d=r+1)$
 \item $X$ is an elliptic normal scroll or an elliptic normal curve $(d=r+1)$.
\end{enumerate}
In particular, $X$ is simply connected if and only if it is from $(a)$, $(b)$, or $(c)$, and $X$ is rationally connected if and only if it is from $(a)$ or $(b)$.
\end{theorem}

We will study adjunction mappings of varieties from the case (1) to obtain a more detailed classification (see Theorem \ref{adj}).

\begin{remark}
(1) The fundamental group of an elliptic scroll or an elliptic curve is $\Z \oplus \Z$.\\[1pt]
(2) Every case of Theorem \ref{main} really occurs. The existence of cases (b) and (c) is clear. See Remark \ref{exisc} for case (d) and see Examples \ref{wfex} and \ref{lfex} for case (a). We note that there are weak Fano varieties of degree $d=r+1$ which are not Fano.\\[1pt]
(3) Every variety from (b), (c), or (d) is projectively normal 
unless it is a non-linearly normal rational scroll. (see Remark \ref{exisc}, Proposition \ref{rocl}, and Remark \ref{ropnrem}).\\[1pt]
\end{remark}

\begin{proof}[Proof of Theorem \ref{main}]
We denote by $H$ a hyperplane section.

By the classification of curves of almost minimal degree, there is nothing to prove when $X$ is a curve. Thus, assume that $n \geq 2$. When $X \subset \P^r$ is a hypersurface, $X$ is simply connected by the Barth-Larsen Theorem (\cite[Corollary 3.2.2]{posI}). In addition, by the adjunction theorem, $X$ is a Fano variety if $d \leq r$, and $X$ is a Calabi-Yau variety if $d = r+1$. Note that a smooth hypersurface $X \subset \P^r$ is rationally connected if and only if $\deg(X) \leq r$, i.e., $X$ is a Fano variety.

We further assume that $e \geq 2$. The case that $X \subset \P^r$ is a Roth variety or a scroll over a curve is already treated in Sections \ref{scr} and \ref{roth}. Thus, suppose that $X$ is neither a Roth variety nor a scroll over a curve. Note that the Veronese surface in $\P^5$ is clearly a Fano variety. By Theorem \ref{nomasa}, we only have to consider the case that the double point divisor $D_{inn}=-K_X + (d-r-1)H$ is semiample (thus, nef). If $d \leq r$, then we may write
$$
-K_X = D_{inn} + (r+1-d)H,
$$
and hence, $-K_X$ is ample, i.e., $X$ is a Fano variety (see also \cite[Corollary 7.6]{N}).

Now, consider the case $d = r+1$. Then, $D_{inn}=-K_X$ is a semiample divisor. We will show that $-K_X$ is big, and hence, $X$ is a weak Fano variety. Note that every weak Fano variety is rationally connected and simply connected (see e.g., \cite[Corollary 1.4]{HM}). If $X \subset \P^r$ is not linearly normal, then it is obtained by an isomorphic projection from $X \subset \P^{r+1}$ with degree $d = r+1$, and hence, it is a Fano variety. Thus, we may assume that $X \subset \P^r$ is linearly normal. We divide into three cases: (1) $e \leq n-2$, (2) $e = n-1$, and (3) $e \geq n$.

If $e \leq n-2$, then by the Barth-Larsen Theorem (\cite[Corollary 3.2.3]{posI}) and Theorem \ref{ne+2}, $-K_X = \ell H$ for some integer $\ell > 0$, and hence, $-K_X$ is ample.

Suppose that $e = n-1$. By the Barth-Larsen Theorem (\cite[Corollary 3.2.2]{posI}), $X$ is simply connected. Note that $d=2n=2e+2$, and $n \geq 3$ since $e \geq 2$. By the classification results by Ionescu (see tables in Introductions of \cite{Io1} and \cite{Io3}), case f) of Theorem I in \cite{Io2} does not occur in our case. Thus, by \cite[Theorem I]{Io2}, $X$ is a Fano variety, a scroll, a linear fibration over a rational surface, or a hyperquadric fibration because $X$ is simply connected and $-K_X$ is non-trivial. Recall that we already exclude scrolls. First, we show that $X$ cannot have a hyperquadric fibration. Suppose that $X$ has a hyperquadric fibration. Let $F$ be a hyperquadric fiber of $X$, and let $H_F:=H|_F$. Then, we have (cf., \cite[p.217]{IT1})
$$
4n = 2n H_F^{n-1} = c_{n-1}(N_{X|\P^{2n-1}}|_F)=c_{n-1}(N_{F|\P^{2n-1}}) = 2 (2n-1),
$$
which is a contradiction. Now, we show that if $X$ has a linear fibration over a rational surface $B$, then it is a weak Fano variety. Since $X \subset \P^r$ is linearly normal and $d=2n$, by \cite[Proposition 4]{IT1}, $B \simeq \P^2$ and $d = \frac{n(n+1)}{2}$. Thus, $n=3$. The only possible $X$ is a Bordiga threefold by the classification (see tables in Introduction of \cite{Io1}). In Example \ref{lfex}, we will give a detailed description of the Bordiga threefold. Note that $X$ is neither a scroll (over a curve) nor a Roth variety, and hence, $-K_X$ is semiample. By the Riemann-Roch Formula, we obtain $(-K_X)^3 = 6$. Thus, $X$ is a weak Fano variety. In particular, we have shown the following.

\begin{lemma}\label{bord}
Let $X \subset \P^{2n-1}$ be an $n$-dimensional non-degenerate smooth projective variety of degree $2n$. Then, it does not admit hyperquadric fibration, and if it has a linear fibration over a surface, then it is the Bordiga threefold.
\end{lemma}

Now, we consider the last case $e \geq n$ ($d=r+1=n+e+1$) and $X$ is not a scroll. Let $C$ be a curve section ($C \subset \P^{e+1}$), and let $g$ be the sectional genus, i.e., the genus of $C$. Denote $H_C := H|_C$. The following bound of the sectional genus plays a crucial role (cf. Lemma 4 of \cite{Io4}).

\begin{lemma}\label{secg}
If $e \geq n$ and $r+1 \geq d$, then $n \geq g$ and $d \geq 2g+1$.
\end{lemma}

\begin{proof}
If $H_C$ is special, then by the Clifford Inequality, we have
$$
e+2 \leq h^0(\mathcal{O}_C(H_C)) \leq \frac{d}{2} +1 \leq \frac{n+e+1}{2} + 1.
$$
Then, we get $2e+4 \leq n+e+1+2$, and hence, $e \leq n-1$, which is a contradiction. Thus, $H_C$ is non-special. By the Riemann-Roch Formula, we have
$$
e+2 \leq h^0(\mathcal{O}_C(H_C)) = d+1-g \leq (n+e+1)+1-g.
$$
Then, we get $g \leq n$ and $g+e+1 \leq d$. Since $g \leq n \leq e$, we obtain $2g+1 \leq g+e+1 \leq d$.
\end{proof}

\begin{remark}
We obtain the same inequality $n \geq g$ by applying the Castelnuovo's Bound on sectional genus.
\end{remark}

If $g=0$, then by Fujita's classification (see e.g., \cite[Theorem A]{Io4}), $X$ is a Fano variety or a rational scroll. Thus, we may assume that $g \geq 1$. Then, we can prove the following.

\begin{lemma}\label{ineq}
If $n \geq g \geq 1$ and $d \geq 2g+1$, then $d + \frac{n}{n-1}\{(n-1)d - 2g +2\} - nd >0$.
\end{lemma}

\begin{proof}
Since $3n > 2g+1$, we have
$$
(2g+1)(n-1)=2gn+n-2g-1 > 2gn-2n=(2g-2)n.
$$
Thus, $n-1 > \frac{2g-2}{2g+1}n$, and hence, there exists a rational number $\epsilon >0$ such that
$$
\frac{2g-2}{2g+1}n+\epsilon = n-1.
$$
Note that
$$
\frac{n-1}{n} + \frac{(n-1)d-2g+2}{d} = \frac{2g-2}{2g+1} + \frac{\epsilon}{n} + (n-1)  + \frac{-2g+2}{d}.
$$
Since $d \geq 2g+1$, we get $\frac{2g-2}{2g+1} \geq \frac{2g-2}{d}$. Then, we have
$$
\frac{n-1}{n} + \frac{(n-1)d-2g+2}{d} = (n-1) + \frac{2g-2}{2g+1} + \frac{-2g+2}{d}  + \frac{\epsilon}{n} > n-1.
$$
By multiplying $\frac{nd}{n-1}$, we obtain $d + \frac{n}{n-1}\{(n-1)d - 2g +2\} > nd$.
\end{proof}

Let $D:=H$ and $E:=\frac{1}{n-1}K_X + H = \frac{1}{n-1}\{K_X + (n-1)H\}.$ By \cite[Theorem 1.4]{Io1}, if $K_X + (n-1)H$ is not base point free, then $X$ is a Fano variety or a scroll. Thus, we may assume that $K_X + (n-1)H$ is base point free (thus, nef), and hence, $E$ is a nef $\Q$-divisor. Note that
$$
2g-2 = \{K_X + (n-1)H\}.H^{n-1} = K_X.H^{n-1} + (n-1)d,
$$
and hence, we get
$$K_X.H^{n-1}=-(n-1)d+2g-2.$$
Now, we have
$$
(D-nE).D^{n-1} = \left( H- \frac{n}{n-1}K_X - nH \right).H^{n-1}=d + \frac{n}{n-1}\{(n-1)d-2g+2\} - nd.
$$
By Lemma \ref{ineq}, we obtain $D^n > nE.D^{n-1}$. Thus, the divisor
$$D-E=H- \left( \frac{1}{n-1}K_X + H \right) = -\frac{1}{n-1}K_X$$
is big by \cite[Theorem 2.2.15]{posI}. We complete the proof of Theorem \ref{main}.
\end{proof}

\section{Cox rings}\label{coxs}

In this section, we study the finite generation of Cox rings of small degree varieties. The purpose is to prove the following theorem.

\begin{theorem}\label{cox}
Let $X \subset \P^r$ be a non-degenerate smooth projective variety of degree $d \leq r+1$.
Then, the Cox ring of $X$ is not finitely generated if and only if $d=r+1$ and $X$ is either a smooth quartic surface with infinite automorphism group or an elliptic scroll.
\end{theorem}

In particular, every non-degenerate smooth projective variety $X \subset \P^r$ of degree $d \leq r$ has a finitely generated Cox ring.

Before giving the proof, we recall the definition of Cox ring. Let $X$ be a regular smooth projective variety (i.e., $h^1(\mathcal{O}_X)=0$). Then, $\Pic(X)$ is finitely generated so that we can choose generators $L_1, \ldots, L_m$ of $\Pic(X)$. The \emph{Cox ring} of $X$ with respect to $L_1, \ldots, L_m$ is defined by
$$
\Cox(X) := \bigoplus_{(a_1, \ldots, a_m) \in \Z^m} H^0( L_1^{\otimes a_1} \otimes \cdots \otimes L_m^{\otimes a_m}).
$$
Note that the finite generation of $\Cox(X)$ is independent of the choice of generators of $\Pic(X)$ (see \cite[Remark in p.341]{HK}), and $\Cox(X)$ is finitely generated if and only if $X$ is a Mori dream space (\cite[Proposition 2.9]{HK}). Typical examples of Mori dream spaces are toric varieties and regular varieties with Picard number one. For further detail on Cox rings and Mori dream spaces, we refer to \cite{HK}.

\begin{example}\label{hypr}
Let $X \subset \P^r$ be a smooth hypersurface. By the Lefschetz Theorem for Picard Group (\cite[Example 3.1.25]{posI}), if $r \geq 4$, then $X$ has finitely generated Cox ring. If $r=2$, then $X$ has finitely generated Cox ring if and only if $X$ is a rational curve. When $r=3$, determining finite generation of Cox ring is a delicate problem even in the case of quartic surfaces. Although very general smooth quartic surfaces have finitely generated Cox rings, there is a smooth quartic surface whose automorphism group is infinite (see e.g., \cite[Theorem 1.2]{Og}) so that its Cox ring is not finitely generated by \cite[Theorems 2.7 and 2.11]{AHL}.
\end{example}

\begin{proposition}\label{rocox}
Let $X \subset \P^r$ be a Roth variety of dimension $n \geq 3$. Then, the Cox ring of $X$ is finitely generated.
\end{proposition}

\begin{proof}
If $b=1$, then $X$ is a rational scroll, which is a toric variety. The Cox ring of a toric variety is a polynomial ring (see e.g., \cite[Corollary 2.10]{HK}). Thus, we only have to consider the case that $b \geq 2$. Recall that $X$ is a divisor of a rational scroll $S = S_{0,0,a_1, \ldots, a_{n-1}}$ with all $a_i \geq 1$, and let $\pi_1 \colon S \rightarrow \P^1$ be the natural projection with a general fiber $F$. Recall that we have a birational morphism $\pi_2 \colon S \rightarrow \overline{S} \subset \P^r$ given by the complete linear system $|H|$, where $\mathcal{O}_{S_{0,0,a_1, \ldots, a_{n-1}}}(H)= \mathcal{O}_{\P_{\P^1}(E)}(1)$. The singular locus of $\overline{S}$ is a line $L \subset \P^r$, which is contained in $X$. There are effective divisors $L_1 \in |H - a_1 F|, \ldots, L_{n-1} \in |H - a_{n-1} F|$ (we can arrange them to be $(\C^{*})^{n+1}$-invariant divisors by the maximal torus action on $S$) such that
$$
\pi_2^{-1}(L) = L_1 \cap \cdots \cap L_{n-1} \simeq L \times \P^1.
$$
Note that $L_1 \cap \cdots \cap L_{n-1} \cap X = L$.

By the Lefschetz Theorem for Picard Group, the map $\Pic(S) \rightarrow \Pic(X)$ is an isomorphism. In particular, the Picard number of $X$ is two. For all $i$, we have an isomorphism
$$
H^0(\mathcal{O}_S(H-a_i F)) \rightarrow H^0(\mathcal{O}_X(H_X - a_i F_X)),
$$
where $F_X = F|_X$ is the restriction, from the exact sequence
$$0 \rightarrow \mathcal{O}_S(-bH - F) \rightarrow \mathcal{O}_S \rightarrow \mathcal{O}_X \rightarrow 0.$$
Let $D_1 := L_1|_X, \ldots, D_{n-1} := L_{n-1}|_X$. We denote by $D_{out}:=-K_X + (d-n-2)H_X$ the double point divisor from outer projection, where $H_X = H|_X$ is the restriction. By \cite[Proposition 3.8]{Il}, $D_{out}.L=0$. Since $D_{out}$ is base point free, we can choose an effective divisor $D_n \in |D_{out}|$ such that $D_n \cap L = \emptyset$. Thus, we obtain
$$
D_1 \cap \cdots \cap D_{n-1} \cap D_n = (L_1 \cap \cdots \cap L_{n-1} \cap X) \cap D_n = L \cap D_n = \emptyset.
$$

On the other hand, let $D_1':=F_1|_X, D_2' :=F_2|_X$, where $F_1$ and $F_2$ are distinct fibers of $\pi_1$. Then, we have
$D_1' \cap D_2' = \emptyset$. Divisor classes of $D_1, \ldots, D_{n-1}$ are in outside of the nef cone of $X$, and the divisor class of $D_n$ is a ray generator of the nef cone of $X$. The other ray generator of the nef cone of $X$ is the divisor class of $F_X$ which is also a ray generator of the effective cone of $X$. Thus, we have
$$
\Cone(D_1, \cdots, D_n) \cap \Cone(D_1', D_2') = \{ 0 \}
$$
By \cite[Theorem 1.3]{It}, $\Cox(X)$ is finitely generated.
\end{proof}

\begin{proposition}\label{rosurcox}
Let $X \subset \P^r$ be a Roth surface of degree $d=b(r-2)+1$. If $b \leq 2$, then the Cox ring of $X$ is finitely generated. In particular, the same conclusion holds when $d \leq r+1$.
\end{proposition}

\begin{proof}
We only have to consider the case $b=2$. By Proposition \ref{roraco}, $X$ is a rational surface. We claim that $-K_X$ is big. Indeed, we have $L=H_X - (r-2)F_X$ and $-K_X = L + F_X$. It is easy to check that $-K_X = \left( \frac{2}{2r-5}L + F_X\right) + \frac{2r-7}{2r-5}L$ is the Zariski decomposition. Since $\left( \frac{2}{2r-5}L + F_X\right)^2 = \frac{4}{2r-5}>0$, it follows that $-K_X$ is big. Then, by \cite[Theorem 1]{TVAV}, the Cox ring of $X$ is finitely generated. The final assertion is trivial since $d \leq r+1$ implies $b \leq 2$.
\end{proof}

\begin{remark}
We do not know whether Cox rings of Roth surfaces are finitely generated in general.
\end{remark}

Now, we come to the proof of Theorem \ref{cox}.

\begin{proof}[Proof of Theorem \ref{cox}]
Recall that the Cox rings of weak Fano varieties and Roth varieties of degree $d \leq r+1$ are finitely generated by \cite[Corollary 1.3.2]{BCHM}, Propositions \ref{rocox} and \ref{rosurcox}. The assertion immediately follows from Theorem \ref{main} and Example \ref{hypr}.
\end{proof}

\section{Adjunction mappings}\label{adjs}

In this section, we study adjunction mappings of weak Fano varieties of small degree. The end product of the minimal model program (see \cite{KM} for basics and \cite{BCHM} for recent progress) should be either a Mori fiber space or a minimal model.
It is natural to study a contraction appeared in the minimal model program for varieties in Theorem \ref{main}.
Note that smooth elliptic curves and Calabi-Yau hypersurfaces are minimal models. We know that scrolls and Roth varieties have natural Mori fibrations. Similarly, the adjunction mapping (if it is defined) also gives a $K_X$-negative contraction of weak Fano varieties. It turns out that the base space of the adjunction mapping are very simple in our case (see Theorem \ref{adj}).

\subsection{Basics of adjunction mappings}

We recall basic notions in adjunction theory (for further detail, see \cite{BS}). Throughout the section, we denote by $H$ a general hyperplane section of $X \subset \P^r$. By \cite[Theorem 1.4]{Io1}, if $K_X + (n-1)H$ is not base point free, then $X$ is a prime Fano variety or a scroll (over a curve) when $n \geq 3$. In the case $n=2$, we should add a Veronese surface $v_2(\P^2)$ in $\P^5$ or $\P^4$ and a quadric hypersurface $Q^2 \subset \P^3$. Now, assume that $K_X + (n-1)H$ is base point free. Then, we can define a surjective morphism $\varphi \colon X \rightarrow B$ given by $|K_X + (n-1)H|$, which is called an \emph{adjunction mapping}.

\begin{proposition}[{\cite[Proposition 1.11]{Io1}}]\label{baadjth}
Let $X \subset \P^r$ be a smooth projective variety of dimension $n$. Assume that $K_X + (n-1)H$ is base point free so that we have the adjunction mapping $\varphi \colon X \rightarrow B$. Then, one of the following holds:
\begin{enumerate}[\indent$(1)$]
 \item $\dim B = 0$, or equivalently, $-K_X = (n-1)H$.
 \item $n \geq 2$ and $\varphi$ gives a hyperquadric fibration over a smooth curve $B$.
 \item $n \geq 3$ and $\varphi$ gives a linear fibration over a smooth surface $B$.
 \item $\dim B = n$.
\end{enumerate}
\end{proposition}

For reader's convenience, we give the complete list of del Pezzo varieties (i.e., $-K_X = (n-1)H$), which was classified by Fujita (\cite{F1} and \cite{F2}; see also \cite[Theorem B]{Io4} and \cite[Section 12.1]{IP}).

\begin{theorem}[Fujita]\label{delpez}
Let $X \subsetneq \P^r$ be a non-degenerate linearly normal smooth del Pezzo variety of dimension $n \geq 2$ and degree $d$. Then, $d=r-n+2$ and $X$ is one of the following:
\begin{enumerate}[\indent$(a)$]
 \item a cubic hypersurface of dimension $n \geq 3$;
 \item a complete intersection of type $(2,2)$;
 \item the Pl\"{u}cker embedding of the Grassmannian $Gr(2,5)$ or its linear section;
 \item the Veronese threefold $v_2(\P^3)$;
 \item a del Pezzo surface embedded by the anticanonical divisor $-K_X$;
 \item the Segre embedding of $\P^2 \times \P^2$ or its hyperplane section;
 \item the Segre embedding of $\P^1 \times \P^1 \times \P^1$;
 \item $\P_{\P^2}(\mathcal{O}_{\P^2}(1) \oplus \mathcal{O}_{\P^2}(2))$ embedded by the tautological line bundle.
\end{enumerate}
\end{theorem}

A del Pezzo surface whose anticanonical divisor is very ample is either $\P^1 \times \P^1$ or the blow-up of $\P^2$ at $r$ points in general position for $r \leq 6$.

\begin{remark}
In the above list, varieties from (a)$\sim$(c) are prime Fano, and varieties from (e)$\sim$(h) have the Picard number $\rho \geq 2$ except for the third Veronese surface $v_3(\P^2) \subset \P^9$ from (e).
\end{remark}

\subsection{Adjunction mappings of weak Fano varieties}\label{adjpf}

The following is the main theorem of this section. This subsection is devoted to proving the following theorem.

\begin{theorem}\label{adj}
Let $X \subset \P^r$ be an $n$-dimensional non-degenerate smooth projective variety of degree $d \leq r+1$, and let $H$ be a general hyperplane section. Assume that $n \geq 2$ and $X$ is a weak Fano variety but not a rational scroll. Then, one of the following holds:
\begin{enumerate}[\indent$(a)$]
 \item $X$ is prime Fano, i.e., $-K_X = \ell H$ for some $\ell >0$ and $\Pic(X) = \Z[H]$.
 \item $X$ is a del Pezzo variety, i.e., $-K_X = (n-1)H$.
 \item $X$ is a Veronese surface $v_2(\P^2)$ in $\P^5$ or $\P^4$, or a quadric hypersurface $\P^1 \times \P^1 \simeq Q \subset \P^3$.
 \item $|K_X + (n-1)H|$ induces a hyperquadric fibration over $\P^1$.
 \item $|K_X + (n-1)H|$ induces a linear fibration over $\P^2$ or $\P^1 \times \P^1$.
\end{enumerate}
In particular, if $X$ is not a Fano variety but a weak Fano variety, then it is a rational variety.
\end{theorem}

We will discuss the classification of prime Fano varieties of degree $d \leq r+1$ (Case $(a)$ of Theorem \ref{adj}) in Subsection \ref{prfano}.

\begin{remark}
(1) Every case of Theorem \ref{adj} really occurs (see Subsection \ref{exs}).\\[1pt]
(2) Not every prime Fano variety is rational. E.g., smooth cubic threefolds in $\P^4$, which are also del Pezzo varieties, are not rational.
\end{remark}

To prove Theorem \ref{adj}, we need some lemmas.

\begin{lemma}[cf. {\cite[Lemma 7]{Io4}}]\label{hqf}
Let $X \subset \P^r$ be a non-degenerate smooth projective variety of degree $d \leq r+1$. Assume that the adjunction mapping $\varphi \colon X \rightarrow C$ induces a hyperquadric fibration over a curve $C$. Then, $C \simeq \P^1$.
\end{lemma}

\begin{proof}
Note that $X$ is not an elliptic scroll. By Theorem \ref{main}, we have $q=h^1(\mathcal{O}_X)=0$, which coincides with the genus of $C$ by \cite[Lemma 6]{Io4}.
\end{proof}


\begin{lemma}[cf. {\cite[Proposition 5]{Io4}}]\label{lf}
Let $X \subset \P^r$ be a non-degenerate smooth projective variety of degree $d \leq r+1$. Assume that the adjunction mapping $\varphi \colon X \rightarrow S$ induces a linear fibration over a surface $S$. Then, $S \simeq \P^2$ or $\P^1 \times \P^1$.
\end{lemma}

\begin{proof}
We will closely follow \cite[Proof of Proposition 5]{Io4}. Note that the case $d \leq r$ is treated in \cite[Proposition 5]{Io4}. We may assume that $d=r+1$. By the Barth-Larsen Theorem (\cite[Corollary 3.2.3]{posI}), we can further assume that $e \geq n-1$. For the case $e = n-1$, we already verified the assertion in Lemma \ref{bord}. Thus, assume that $e \geq n$. Since $X$ is not an elliptic scroll, we have $q=h^1(\mathcal{O}_X)=0$.

Let $S' = X \cap H_1 \cap \cdots \cap H_{n-2}$ be a smooth surface, where each $H_i$ is a generic hyperplane section. By the adjunction formula and Lemma \ref{secg}, we have
$$
2g-2 = H_{S'}^2 + H_{S'}.K_{S'} = d + H_{S'}.K_{S'} \geq 2g+1 + H_{S'}.K_{S'},
$$
and hence, $H_{S'}.K_{S'}<0$. In particular, $h^0(\mathcal{O}_{S'}(K_{S'}))=h^2(\mathcal{O}_{S'})=0$. Now, observe that $n \geq 3$. We have the following exact sequence
$$
0 \rightarrow \mathcal{O}_X (K_X + (n-2)H) \rightarrow \mathcal{O}_X (K_X + (n-1)H) \rightarrow \mathcal{O}_H (K_H + (n-2)H_H) \rightarrow 0.
$$
Since $K_X + (n-1)H$ is not big, $h^0(\mathcal{O}_X (K_X + (n-2)H))=0$. Furthermore, by Kodaira Vanishing, $h^1(\mathcal{O}_X (K_X + (n-2)H))=0$. Let $g$ be the sectional genus of $X \subset \P^r$. Then, we have
$$
h^0(\mathcal{O}_X(K_X + (n-1)H)) = h^0(\mathcal{O}_H(K_H + (n-2)H_H)) = \cdots = h^0(\mathcal{O}_{S'}(K_{S'} + H_{S'})),
$$
and by \cite[Lemma 1.1]{Io1}, $h^0(\mathcal{O}_{S'}(K_{S'} + H_{S'}))=g$. Thus, we get $\varphi \colon X \rightarrow S \subset \P^{g-1}$. Let $H_S$ be a generic hyperplane section of $S \subset \P^{g-1}$, and let $Y := \varphi^{-1}(H_S)$. Note that $Y$ is a scroll of dimension $n-1$ and degree $d_Y = (K_X + (n-1)H).H^{n-1}=2g-2$ over the curve $H_S$. Let $m$ be the dimension of the smallest linear subspace of $\P^r$ containing $Y$, i.e., $Y \subset \P^m$ is non-degenerate. By Lemma \ref{scrcod}, $m \geq 2(n-1)-1 = 2n-3$. By Lemma \ref{secg}, $m \geq 2n-3 \geq 2g-3 = d_Y -1$. Thus, by Theorem \ref{minsc}, the genus $g'$ of $H_Y$ is 0 or 1. Suppose that $g'=1$. By Lemma \ref{scrcod}, we must have $m \geq 2(n-1)$. It follows that $m \geq 2n-2 \geq 2g-2 = d_Y$, which is a contradiction to Theorem \ref{minsc}. Thus, $g'=0$. By Fujita's classification (\cite[Theorem A]{Io4}), $S \simeq \P^2$ (in this case, we have $g= \Delta(X, H) =3$ or $g=6$) or $S$ is a scroll over $\P^1$, i.e., Hirzebruch surface.

It suffices to show that if $S$ is a Hirzebruch surface, then $S \simeq \P^1 \times \P^1$. Suppose that $S \simeq F_a := \P_{\P^1}(\mathcal{O}_{\P^1} \oplus \mathcal{O}_{\P^1}(-a))$ be a Hirzebruch surface for some integer $a \geq 0$. Then, $H_S = C + bF$, where $C$ is a section with $C^2=-a$ and $F$ is a general fiber of the projection $F_a \rightarrow \P^1$, such that $b > a$. Let $Y_0 := \varphi^{-1}(C)$ and $Y_1 := \varphi^{-1}(F)$ be $(n-1)$-dimensional rational scrolls. For each $i$, let $m_i$ be the dimension of the smallest linear subspace of $\P^r$ containing $Y_i$, and let $d_i$ be the degree of $Y_i \subset \P^{m_i}$. Then, by Lemma \ref{scrcod}, $m_i \geq 2(n-1)-1 = 2n-3$. Thus, $d_i \geq m_i - (n-1) + 1 \geq n-1$. It follows that
$$
2g-2 = d_Y = d_0 + b d_1 \geq d_0 + d_1 \geq 2(n-1),
$$
and hence, we get $g \geq n$. By Lemma \ref{secg}, we obtain $g=n$. Thus, we must have $b=1$, and hence, $a=0$, i.e., $S = \P^1 \times \P^1$. In this case, we have $n=g = h^0(\mathcal{O}_S(H_S))=4$. (By the same argument, we can show that if $S \simeq \P^2$ and $g=6$, then $n=6$.)
\end{proof}

\begin{lemma}\label{fano}
Let $X \subset \P^{2n-1}$ be a non-degenerate smooth projective variety of dimension $n \geq 3$ and degree $d \leq 2n$, and let $H$ be a general hyperplane section. If $-K_X = (n-2)H$, then $X$ is a prime Fano variety.
\end{lemma}

\begin{proof}
By Ionescu's classification (see tables in Introductions of \cite{Io1} and \cite{Io3}), $X$ is a complete intersection provided that $n \leq 4$. Thus, assume $n \geq 5$. We claim that the Picard number $\rho(X)$ of $X$ is $1$. By \cite[Theorem A]{W1}, if $n \geq 7$, then $\rho(X)=1$. Consider the case $n=6$. By \cite[Theorem B]{W2}, either $\rho(X)=1$ or $X \simeq \P^3 \times \P^3$. However, $\P^3 \times \P^3$ cannot have an embedding in $\P^{11}$ with degree $d \leq 12$ so that $\rho(X)=1$ for $n=6$. Suppose that $n=5$. A general hyperplane section $Y$ of $X$ is a Fano fourfold with $-K_Y = 2H_Y$ and $H_Y^4 = 10$. Clearly, a general member of $|H_Y|$ is irreducible and smooth. Thanks to Wi\'{s}niewski's classification (\cite{W2}; see also \cite[Secion 12.7]{IP}), we see that there is no such Fano fourfold with Picard number $\rho(Y) \geq 2$. By the Lefschetz Theorem for Picard Group, we obtain $\rho(X)=\rho(Y)=1$. We have shown the claim.

Since the Picard group of a Fano variety is torsion-free (\cite[Proposition 2.1.2]{IP}), $\Pic(X)$ is generated by an ample divisor $L$. Let $H = mL$ for some integer $m \geq 1$. Then, we have $d=H^n = m^n L^n \leq 2n$. Since $n \geq 5$, we must have $m=1$. Thus, $X$ is prime Fano.
\end{proof}

We are ready to prove Theorem \ref{adj}.

\begin{proof}[Proof of Theorem \ref{adj}]
Let $e = r-n$ be the codimension. We divide into three cases. Firstly, assume that $e \leq n-2$. By the Barth-Larsen Theorem (\cite[Corollary 3.2.3]{posI}), $X$ is a prime Fano variety. Secondly, assume that $e=n-1$ (then, $d \leq 2e+2$). By Ionescu's classification (see tables in Introductions of \cite{Io1} and \cite{Io3}), we can easily verify the assertion for $n \leq 4$. For $n \geq 5$, by \cite[Theorem I]{Io2}, the assertion follows from Lemmas \ref{hqf}, \ref{lf}, and \ref{fano}. Finally, suppose that $e \geq n$ (then, $d \leq 2e+1$). We denote by $g$ the sectional genus. By Lemma \ref{secg}, $g \leq n$. If $g \leq 1$, then the assertion follows from Fujita's classification (see \cite[Proposition 2.4]{Io1} and \cite[Theorems A and B]{Io4}). If $g=2$, then $X$ has a hyperquadric fibration over $\P^1$ by \cite[Corollary 3.3]{Io1} for $n \geq 3$ and the Castelnuovo's Theorem (see e.g., \cite[Proposition 3.1]{Io1}) for $n=2$. If $g=3$ (resp. $g=4$), then the assertion follows from \cite[Theorem 4.2]{Io1} (resp. \cite[Theorem 11.6.3]{BS} together with Lemmas \ref{hqf} and \ref{lf}). It remains the case $5 \leq g \leq n$. If $K_X + (n-1)H$ is not base point free, then $X$ is a prime Fano variety by \cite[Theorem 1.4]{Io1}. Now, suppose that we can define the adjunction mapping $\varphi \colon X \rightarrow B$. By \cite[Theorem I]{Io2}, we only have to consider the cases $\dim B = 1$ or $2$, and then, the remaining part immediately follows from Lemmas \ref{hqf} and \ref{lf}.
\end{proof}

\begin{remark}
We can analyze in detail the cases (d) and (e) in Theorem \ref{adj}. If $X$ is from (d), then it is a divisor of a rational scroll (\cite[Lemma 6]{Io4}). If $X$ is from (e) and the base is $\P^2$, then $g = \Delta(X, H)=3$ (e.g., Bordiga threefold in Example \ref{lfex}) or $n=g =6$ (e.g., Segre embedding of $\P^2 \times \P^4$). If $X$ is from (e) and the base is $\P^1 \times \P^1$, then $n=g=\Delta(X,H)=4$, and by \cite[Theorem 6.3]{LN}, there is only one case (see Example \ref{fiex}).
\end{remark}

\subsection{Examples}\label{exs}

We give examples of weak Fano varieties having fibrations coming from adjunction mappings.

\begin{example}\label{wfex}
(1) Let $F_2 := \P_{\P^1}(\mathcal{O}_{\P^1} \oplus \mathcal{O}_{\P^1}(-2))$ be a Hirzebruch surface with the fibration $f \colon F_2 \rightarrow \P^1$. Then, $F_2$ is not a Fano surface but a weak Fano surface. For the embedding $X:= \varphi_{|H|}(F_2) \subset \P^{11}$ given by a very ample divisor $H:=2C + 5F$, where $F$ is a fiber of $f$ and $C$ is a section of $f$ with $C^2=-2$, we have $H^2 = 12$, i.e., $\deg(X)=12$. Note that $K_X + H \sim F$ and $F.H=2$, and hence, the adjunction mapping $\varphi_{|K_X+H|}$ gives a hyperquadric fibration over $\P^1$.\\[1pt]
(2) Let $\pi \colon \widetilde{F}_2 \rightarrow F_2$ be the blow-up at a point not in $C$, and let $E$ be the exceptional divisor. Then, $\widetilde{F}_2$ is also not a Fano surface but a weak Fano surface. Furthermore, the very ample divisor $H':=\pi^{*}(H)-E$ gives an embedding $X' :=\varphi_{|H'|}(\widetilde{F}_2) \subset \P^{10}$ with $\deg(X')=11$. Note that $K_{X'}+H' \sim \pi^{*}F$ and $\pi^{*}F.H'=2$, and hence, the adjunction mapping $\varphi_{|K_{X'}+H'|}$ gives a hyperquadric fibration over $\P^1$ with a singular fiber.
\end{example}

\begin{example}\label{lfex}
There exists a stable vector bundle $E$ of rank $2$ on $\P^2$ such that it is given by an extension
$$
0 \rightarrow \mathcal{O}_{\P^2} \rightarrow E \rightarrow \mathcal{I}_Y (4) \rightarrow 0,
$$
where $Y$ is a closed subscheme of $\P^2$ consisting of $10$ distinct points, $c_1(E)=4, c_2(E)=10$, and $E|_L \simeq \mathcal{O}_L(2) \oplus \mathcal{O}_L(2)$ when $L$ is a generic line (see \cite[Proposition 7.5]{Io1} and \cite{Ot}). Then, for $X :=\P_{\P^2} (E) \subset \P^5$, we have $\deg(X)=6$, and the adjunction mapping induces a linear fibration over $\P^2$. This $X$ is called the \emph{Bordiga threefold}. Now, we show that $X$ is a weak Fano variety but not a Fano variety. Recall that $X$ is a weak Fano variety with $(-K_X)^3=6$. By the classification of Mori and Mukai (\cite{MM}; see also \cite[Section 12.3]{IP}), there is no rational Fano threefold with $(-K)^3=6$, and hence, $X$ is not a Fano variety.
\end{example}

More examples having fibrations over $\P^1$ or $\P^2$ can be found in \cite{Io4}.

\begin{example}\label{fiex}
Let $Q:= \P^1 \times \P^1$ and $E:=\mathcal{O}_{Q}(1,1)^{\oplus 3}$. Consider the embedding $X := \P_{Q}(E) \subset \P^{11}$ given by the complete linear system $|\mathcal{O}_{\P_{Q}(E)}(1)|$. Then, $\deg(X)=12$. Furthermore, $X$ is a Fano variety, and the adjunction mapping induces a linear fibration over a quadric hypersurface in $\P^3$.
\end{example}

\subsection{Prime Fano varieties of small degree}\label{prfano}

Here, we further investigate the prime Fano case in Theorem \ref{adj}. Let $X \subset \P^r$ be a non-degenerate prime Fano variety of dimension $n$, codimension $e$, and degree $d \leq r+1$. If $e \leq \frac{n-1}{2}$, then Hartshorne conjectured that $X$ must be a complete intersection (\cite{Ht}). If $e \geq n$, then by \cite[Theorem I]{Io2}, $X$ is a del Pezzo variety. It only remains the case $e+1 \leq n \leq 2e$. We may write $-K_X = \ell H$ for some integer $\ell >0$, where $H$ is the ample generator of $\Pic(X)$. If $\ell \geq n-1$, then $(X, \mathcal{O}_X(H)) \simeq (\P^n, \mathcal{O}_{\P^n}(1))$, $X$ is a hyperquadric in $\P^{n+1}$, or $X$ is a del Pezzo variety (see e.g., \cite[Corollary 2.1.14]{IP}). Thus, we may assume that $\ell \leq n-2$.

\begin{proposition}
Let $X \subset \P^r$ be a non-degenerate prime Fano variety of dimension $n$, codimension $e$, and degree $d$ with $-K_X = \ell H$. If $e+1 \leq n \leq 2e$, $d \leq r+1$, and $\ell \leq n-2$, then $\ell = n-2$.
\end{proposition}

\begin{proof}
By \cite[Theorem A, Theorem B, and Proposition 10]{Io4}, we only have to consider the case $d=r+1$. By the adjunction formula, we have
\begin{equation}\label{2g-2}
2g-2 = (n- \ell -1)d,
\end{equation}
where $g$ is the sectional genus of $X$. Recall the Castelnuovo's Bound:
\begin{equation}\label{casb}
g \leq \frac{m(m-1)}{2}e + m \epsilon,
\end{equation}
where $m = \lfloor \frac{d-1}{e} \rfloor$ and $\epsilon = d - me -1$. Suppose that $n=2e$ ($d=3e+1$). Then, $m=3$ and $\epsilon =0$, and hence, by (\ref{casb}), we get $g \leq 3e$. By (\ref{2g-2}), we have $(n- \ell - 1)(3e+1) \leq  6e -2$, so we obtain $n - \ell - 1 \leq 1$, i.e., $\ell \geq n-2$. Now, suppose that $n \leq 2e-1$ ($d=n+e+1$). Then, $m=2$ and $\epsilon = n-e$, and hence, by (\ref{casb}), we get $g \leq 2n-e$. By (\ref{2g-2}), we have $(n - \ell - 1)(n+e+1) \leq 4n - 2e -2$. If $n - \ell -1 \geq 2$, then $2e + 2 \leq n$, which is a contradiction. Thus, $\ell \geq n-2$.
\end{proof}

For reader's convenience, we give the complete list of Mukai varieties (i.e., Fano varieties with $-K_X = (n-2)H$ and $\Pic(X)=\Z[H]$), which are completely classified by Mukai (\cite{M}; see also \cite[Theorem 5.2.3]{IP}), of degree $d \leq r+1$.

\begin{theorem}[Mukai]
Let $X \subset \P^r$ be a non-degenerate Mukai variety of dimension $n$ and degree $d$. If $d \leq r+1$, then $X$ satisfies one of the following:
\begin{enumerate}[\indent$(a)$]
 \item a complete intersection of type $(2,3)$ and $n \geq 3$ or of type $(2,2,2)$ and $n \geq 4$;
 \item $n=5,6, d=10$, and the intersection $C \cap Q \subset \P^{10}$, where $C \subset \P^{10}$ is the cone over the Pl\"{u}cker embedding of the Grassmannian $Gr(2,5) \subset \P^9$ and $Q \subset \P^{10}$ is a quadric hypersurface, or its hyperplane section;
 \item $6 \leq n \leq 10, d=12$, and the spinor variety $X \subset \P^{15}$ or its linear sections;
 \item $n=7,8, d=14$, and the Pl\"{u}cker embedding of the Grassmannian $Gr(2,6) \subset \P^{14}$ or its hyperplane section.
\end{enumerate}
\end{theorem}

\section{Smooth projective varieties of degree at most $r+2$}\label{r+2s}

In this section, we consider the classification of smooth projective varieties $X \subset \P^r$ of degree $d \leq r+2$. We first prove the following theorem.

\begin{theorem}\label{r+2}
Let $X \subset \P^r$ be an $n$-dimensional non-degenerate smooth projective variety of degree $d \leq r+2$. Then, one of the following holds:
\begin{enumerate}[\indent$(a)$]
 \item $X$ is rationally connected. $($If $e \geq n+1 \geq 4$, then $-K_X$ is big.$)$
 \item $X$ is a Roth variety but not a hypersurface or a rational scroll.
 \item $X$ is a Calabi-Yau hypersurface with $n \geq 2$ $(d=r+1)$ or a K3 surface in $\P^4$ $(d=6)$.
 \item $X$ is a hypersurface of general type with $n \geq 2$ $(d=r+2)$.
 \item $X$ is a curve of genus $g =1,2$, a plane quartic curve, or an elliptic scroll.
\end{enumerate}
In particular, $X$ is simply connected if and only if it is from $(a)$, $(b)$, $(c)$, or $(d)$, and $X$ is rationally connected if and only if it is from $(a)$ or $(b)$.
\end{theorem}

\begin{proof}
By Theorem \ref{main}, we may assume that $X$ is linearly normal and $d=r+2$. 
When $X$ is a curve, the assertion is trivial. Thus assume that $n \geq 2$.
If $X$ is a scroll over a curve of genus $g$, then the assertion follows from Proposition \ref{minsc}.
Let $e := r-n$ be the codimension of $X \subset \P^r$, and let $H$ be a general hyperplane section. If $e=1$, then there is nothing to prove. Now, assume that $e \geq 2$ and $X$ is neither a scroll nor a Roth variety. By Theorem \ref{nomasa}, $D_{inn} = -K_X + H$ is semiample. If $e \leq n-2$, then by the Barth-Larsen Theorem, $X$ must be a Fano variety or a Calabi-Yau variety. If $n=e=2$ ($d=6$), then we can verify the assertion by the classification result (see \cite[table in Introduction]{Io1}). We note that there exists a K3 surface $S \subset \P^4$ of degree 6. Assume that $e \geq n-1$. Since $d=n+e+2 \leq en+1$ unless $e=n=2$, it follows that $X$ is uniruled by Theorem \ref{ne+2}. On the other hand, if $n \geq 3$ and $e \leq n-2$, then $X$ is also uniruled. Since uniruled varieties cannot be Calabi-Yau, we conclude that $X$ is uniruled or from $(c)$. 

By Corollary \ref{rora}, varieties from (b) are rational.
Furthermore, note that every variety from $(c)$ is simply connected. In addition, by the Lefschetz Hyperplane Theorem (\cite[Theorem 3.1.21]{posI}), varieties from $(d)$ are also simply connected.

Now, suppose that $X$ is not from $(b)$, $(c)$, $(d)$, or $(e)$. We prove that $X$ is rationally connected, and hence, it is simply connected. If $e \leq n-2$, then we already saw that $X$ is Fano, and hence, it is rationally connected. If $e=n-1$, then $X$ is simply connected by the Barth-Larsen Theorem (\cite[Corollary 3.2.2]{posI}). In this case, $d=2n+1$, and $\Delta(X, H)=n+1$. The case $n=3$ can be directly checked by the classification result (see \cite[table in Introduction]{Io1}). If $n \geq 4$, then $d > \frac{3}{2}\Delta(X,H)+1$. By \cite[Corollary 4.5]{Fu}, we only have to show that if $X$ has a hyperquadric fibration over a curve $C$ or it has a linear fibration over a surface $S$ with $h^0(\mathcal{O}_S(K_S))=0$, then $X$ is rationally connected. In the first case, $C$ is rational by \cite[Lemma 6]{Io4}. In the second case, $S$ is also rational by \cite[Proposition 4]{IT1}. Thus, in both cases, $X$ is rational.

We assume that $e \geq n$. First, we show that $h^1(\mathcal{O}_X)=0$. We claim that $-K_X + H$ is big except when a general surface section of $X$ is a K3 surface, which is simply connected. In the exceptional case, $X$ is simply connected by the Lefschetz Hyperplane Theorem.
Let $g$ be the sectional genus of $X \subset \P^r$. By the Castelnuovo's Bound for sectional genus, we have
$g \leq n+2$ for $n=e$, and $g \leq n+1$ for $n \leq e-1$. In the first case, we have $d=2e+2$. If $g=e+2$, then the curve section of $X$ is a Castelnuovo curve, and hence, by \cite[p.67]{Hr} or \cite[Proposition 3.13]{Il}, the surface section of $X$ is also a Castelnuovo surface (see Definition \ref{casvar}). By Lemma \ref{mincas}, this surface is K3. Thus, we can assume that $g \leq n+1$ so that $2g-2 \leq 2n < n+e+2=d$. On the other hand, let $D:=nH$ and $E:=K_X + (n-1)H$. By \cite[Theorem 1.4]{Io1}, we may assume that $E$ is nef. We have
$$
nE.D^{n-1} = n^n (K_X + (n-1)H).H^{n-1} = n^n (2g-2) < n^n d = D^n.
$$
By \cite[Theorem 2.2.15]{posI}, $D-E=-K_X + H$ is big.

We proceed the induction on $n$ to show that $h^1(\mathcal{O}_X)=0$ provided that $X$ is not a scroll over a curve. The case $n=2$ can be done by case-by-case analysis as follows. We already showed the assertion when $n=e=2$. When $e \geq 3$, we have $g \leq 3$. Then, the assertion follows from \cite[Theorem A]{Io4} for $g=0$, \cite[Proposition 2.6]{Io1} for $g=1$, \cite[Proposition 3.1]{Io1} for $g=2$, and \cite[Theorem 4.1]{Io1} for $g=3$. Suppose that $n \geq 3$. By the Bertini Theorem, $H \subset \P^{r-1}$ is a non-degenerate smooth projective variety of degree $d = (r-1)+3$. Note that $H$ should not be a scroll over a curve. We have the exact sequence
\begin{equation}\label{exact}
0 \rightarrow \mathcal{O}_X (-H) \rightarrow \mathcal{O}_X \rightarrow \mathcal{O}_H \rightarrow 0.
\end{equation}
Note that $-K_X + H$ is nef and big and $H = K_X + (-K_X + H)$, and hence, by Kawamata-Viehweg Vanishing, $h^1(\mathcal{O}_X(H))=0$. From (\ref{exact}) with twisting $\mathcal{O}_X(H)$, $h^1(\mathcal{O}_X)=0$ if and only if the natural map $H^0(\mathcal{O}_X(1)) \rightarrow H^0(\mathcal{O}_H(1))$ is surjective. Suppose that $H^0(\mathcal{O}_X(1)) \rightarrow H^0(\mathcal{O}_H(1))$ is not surjective. Since $X \subset \P^r$ is linearly normal, $H \subset \P^{r-1}$ must not be linearly normal. Then, $H \subset \P^{r-1}$ can be obtained by an isomorphic projection from $H \subset \P^r$ with degree of $H$ is $r+2$. By induction hypothesis, $H$ is regular. By Kodaira Vanishing, $h^1(\mathcal{O}_X(-H))=h^2(\mathcal{O}_X(-H))=0$. Thus, from (\ref{exact}), we have $h^1(\mathcal{O}_X)=h^1(\mathcal{O}_H)=0$.

We still assume that $X$ is not from $(b)$, $(c)$, $(d)$, or $(e)$, and $e \geq n$. Then, we have $d = n+e+2 \leq 2e+2$. By \cite[Theorem I]{Io2}, we only have to consider the case that the adjunction mappings of $X$ induce hyperquadric fibration over a curve $C$ or linear fibrations over a ruled surface $S$. In the first case, by \cite[Lemma 6]{Io4}, $C$ is a rational curve. In the second case, $h^1(\mathcal{O}_X)=h^1(\mathcal{O}_S)=0$, and hence, $S$ is a rational surface. Thus, in both cases, $X$ is rational.

It remains to show that $-K_X$ is big when $e \geq n+1 \geq 4$. Recall that $g$ is the sectional genus of $X$. We proved that $n+1 \geq g$, which implies that $d \geq 2g+1$. As in Lemma \ref{ineq}, it follows that $d + \frac{n}{n-1}\{(n-1)d - 2g +2\} - nd >0$, and hence, we can conclude that $-K_X$ is big by \cite[Theorem 2.2.15]{posI} except when $n=3, e=4$ and $g=4$ ($d=9$). By the classification of degree nine varieties (see \cite[Proposition 3.1]{FL}), we can easily check that $-K_X$ is also big for the exceptional case.
\end{proof}

By the same arguments in Section \ref{adjs}, we can prove the analogous statement to Theorem \ref{adj}. There are some exceptional cases for $n \leq 4$ which can be completely classified. We leave the proof to the interested reader.

\begin{theorem}\label{r+2adj}
Let $X \subset \P^r$ be an $n$-dimensional non-degenerate smooth projective variety of degree $d = r+2$, and let $H$ be a general hyperplane section. Assume that $n \geq 5$ and $X$ is from the case $(a)$ in Theorem \ref{r+2}. Then, one of the following holds:
\begin{enumerate}[\indent$(a)$]
 \item $X$ is prime Fano, i.e., $-K_X = \ell H$ for some $\ell >0$ and $\Pic(X) = \Z[H]$.
 \item $|K_X + (n-1)H|$ induces a hyperquadric fibration over $\P^1$.
 \item $|K_X + (n-1)H|$ induces a linear fibration over a smooth del Pezzo surface.
\end{enumerate}
In particular, if $X$ is not a prime Fano, then it is a rational variety.
\end{theorem}

Finally, we give examples of varieties from $(a)$ in Theorem \ref{r+2}.

\begin{example}
(1) Let $\pi \colon S \rightarrow \P^2$ be the blow-up at 9 points in general position, and let $E_1, \ldots, E_{9}$ be the exceptional divisors. Then, the very ample divisor $H := \pi^{*}(4L) - E_1 - \cdots - E_9$ gives an embedding $S \subset \P^5$ with $\deg(S)=7$. where $L$ is a line in $\P^2$. Note that $-K_S$ is not big. Moreover, the Cox ring of $S$ is not finitely generated because there are infinitely many ($-1$)-curves.\\[1pt]
(2) Let $X \subset \P^5$ be a \emph{Palatini threefold} (see e.g., \cite{Ot}). Then, $\deg(X)=7$ and $X$ has a linear fibration over a cubic surface in $\P^3$.\\[1pt]
(3) There is a threefold $X \subset \P^5$ of degree $7$ such that $|K_X + H|$ induces a cubic hypersurface fibration over $\P^1$ (see \cite[table in Introduction]{Io1}).
\end{example}

\section{Smooth projective varieties of degree at most $ne+2$}\label{CY}

In this section, we prove Theorem \ref{ne+2}.
For this purpose, we need to classify Castelnuovo varieties of minimal degree. See \cite{Hr} for more detail on (possibly singular) Castelnuovo varieties.

\begin{definition}[{\cite[p.44]{Hr}}]\label{casvar}
Let $X \subset \P^r$ be a smooth projective variety of dimension $n$, codimension $e$, and degree $d$. Then, $X$ is called a \emph{Castelnuovo variety} if $d \geq ne+2$ and
$$
p_g(X)=h^0(\mathcal{O}_X(K_X))={m \choose n+1}e + {m \choose n} \epsilon,
$$
where $m = \lfloor \frac{d-1}{e} \rfloor$ and $\epsilon = d - me -1$.
\end{definition}

\begin{lemma}\label{mincas}
Let $X \subset \P^r$ be a non-degenerate smooth Castelnuovo variety of dimension $n$, codimension $e$, and degree $d$. If $d=ne+2$, then $X$ is Calabi-Yau.
\end{lemma}

\begin{proof}
Note that every Castelnuovo variety is arithmetically Cohen-Macaulay (see \cite[p.66]{Hr}), and in particular, $h^i(\mathcal{O}_X)=0$ for $0 < i < n$. According to Harris' classification (see \cite[p.65]{Hr}), one of the following holds: (1) $X$ is supported on a rational normal scroll $S$, (2) $X$ is a complete intersection of type $(2,n+1)$ (since $d=2(n+1)$), or (3) $X$ is a divisor of a cone over the Veronese surface in $\P^5$. Suppose that we are in Case (1). Recall that $X \in |-K_S|$ when $d=ne+2$ (see \cite[p.56]{Hr}), and thus, $\mathcal{O}_X(K_X) = \mathcal{O}_X$ by the adjunction formula, i.e., $X$ is Calabi-Yau. If we are in Case (2), then by the adjunction formula, $X$ is Calabi-Yau. Since $d=4n+2$ is not divisible by $4$ (see \cite[p.64]{Hr}), Case (3) cannot occur.
\end{proof}

Now, we prove Theorem \ref{ne+2}.

\begin{proof}[Proof of Theorem \ref{ne+2}]
In \cite[Proof of Corollary 1.6]{Z}, using the Castelnuovo's bound on the sectional genus and the concavity theorem (see e.g., \cite[Theorem 1.1]{Z} or \cite[Example 1.6.4]{posI}), Zak showed that $K_X.H^{n-1} <0$ (resp. $K_X.H^{n-1} \leq 0$) provided that $d \leq ne+1$ (resp. $d\leq ne+2$).
It is well known that $K_X.H^{n-1}<0$ implies the uniruledness of $X$.
Thus, it only remains to consider the case $H^n=d=ne+2$ and $K_X.H^{n-1}=0$. Let $g$ be the sectional genus of $X \subset \P^r$. Then, we have
$$
2g-2 = (K_X + (n-1)H).H^{n-1} =(n-1)H^n =(n-1)(ne+2) = 2{m \choose 2}e + 2{m \choose 1}\epsilon-2,
$$
where $m=n$ and $\epsilon = 1$, and hence, the generic curve section $C$ is a Castelnuovo curve. By \cite[p.67]{Hr} or \cite[Proposition 3.13]{Il}, $X$ is a Castelnuovo variety, and hence, the assertion follows from Lemma \ref{mincas}.
\end{proof}

\begin{remark}\label{enrem}
There is a non-degenerate Enriques surface $S \subset \P^r$ of codimension $e$ and degree $d = 2e+3$. Note that $S$ is neither Castelnuovo, Calabi-Yau, nor uniruled (see \cite[Proposition 2.1 and Corollary 3.2]{Bui}).
\end{remark}

The following theorem, which was pointed out by the referee, is a generalization of Theorem \ref{ne+2}.

\begin{theorem}
Let $X \subset \P^r$ be a non-degenerate smooth projective variety of dimension $n$, codimension $e$, and degree $d$. Fix an integer $l$ with $0 \leq l \leq n-1$. If $d < (n-l)e+2$ $($resp. $d \leq (n-l)e+2$$)$, then $-K_X.H^{n-1} > ld$ $($resp. $-K_X.H^{n-1} \geq ld$$)$.
\end{theorem}

\begin{proof}
Consider an $(n-l)$-dimensional general linear section $Y$ of $X \subset \P^r$. By the adjunction formula,
$K_X.H^{n-1}+ld=(K_X+lH).H^{n-1} = K_Y.(H|_Y)^{n-l-1}$. By applying \cite[Proof of Corollary 1.6]{Z} to $Y$, we get the assertion.
\end{proof}


\end{document}